\documentclass[11pt]{amsart}

\usepackage{amsmath}
\usepackage{amsthm}
\usepackage{amsfonts}
\usepackage{amssymb}

\usepackage{graphics}
\usepackage{epsfig}
\usepackage{color}
\usepackage{verbatim}

\newcommand\R{{\mathbb{R}}}
\newcommand\C{{\mathbb{C}}}

\newcommand\tr{\operatorname{tr}}

\newcommand\Mat{\operatorname{Mat}}
\newcommand\var{{{\mathbf{V}}}}
\newcommand\muM{{\mu_{\mathcal{M}}}}
\newcommand\muG{{\mu_{\mathcal{G}}}}
\newcommand\muV{{\mu_{\mathcal{V}}}}
\newcommand\supp{\mathtt{S}}

\def\M#1{{\mathcal{M}}_{d,#1}}
\def\G#1{{\mathcal{G}}_{d,#1}}
\def\V#1{{\mathcal{V}}_{d,#1}}

\theoremstyle{plain}
  \newtheorem{theorem}{Theorem}[section]
  
  \newtheorem{proposition}[theorem]{Proposition}
  \newtheorem{lemma}[theorem]{Lemma}

\theoremstyle{definition}

  \newtheorem*{remark}{Remark}

\title[A sharp $k$-plane Strichartz inequality]{A sharp $k$-plane Strichartz inequality for the Schr\"odinger equation }
\author{J. Bennett, N. Bez, T. C. Flock, S. Guti\'errez and M. Iliopoulou}
\thanks{This work was supported by the European Research Council [grant
number 307617] (Bennett, Flock, Iliopoulou) and JSPS Grant-in-Aid for Young Scientists (A) [grant number 16H05995] (Bez)}

\address{Jonathan Bennett, Taryn C. Flock, and Susana Guti\'errez: School of Mathematics, The Watson Building, University of Birmingham, Edgbaston,
Birmingham, B15 2TT, England.}
\email{J.Bennett@bham.ac.uk, T.C.Flock@bham.ac.uk, s.gutierrez@bham.ac.uk }
\address{Neal Bez: Department of Mathematics, Graduate School of Science
and Engineering, Saitama University, Saitama 338-8570, Japan}
\email{nealbez@mail.saitama-u.ac.jp}
\address{Marina Iliopoulou:  Department of Mathematics, University of California, Berkeley, CA 94720-3840, USA
}
\email{m.iliopoulou@berkeley.edu}

\begin{document}

\begin{abstract}
We prove that 
$$
\|X(|u|^2)\|_{L^3_{t,\ell}}\leq C\|f\|_{L^2(\mathbb{R}^2)}^2,
$$ where $u(x,t)$ is the solution to the linear time-dependent Schr\"odinger equation on $\mathbb{R}^2$ with initial datum $f$, and $X$ is the (spatial) X-ray transform on $\mathbb{R}^2$. In particular, we identify the best constant $C$ and show that a datum $f$ is an extremiser if and only if it is a gaussian. We also establish bounds of this type in higher dimensions $d$, where the X-ray transform is replaced by the $k$-plane transform for any $1\leq k\leq d-1$. In the process we obtain sharp $L^2(\mu)$ bounds on Fourier extension operators associated with certain high-dimensional spheres, involving measures $\mu$ supported on natural ``co-$k$-planarity" sets.
\end{abstract}
\maketitle

\section{Introduction}\label{sec:intro}

The purpose of this paper is to expose a natural interplay between the solution to the time-dependent free Schr\"odinger equation on $\mathbb{R}^d$ and the (spatial) $k$-plane transform for $1\leq k\leq d-1$. In the interests of exposition we begin by describing our results in $\mathbb{R}^2$, where matters are particularly simple, and where the $k$-plane transform reduces to the classical X-ray transform.

Let $u(x,t)$ be the solution to the linear time-dependent two-dimensional Schr\"odinger equation
\begin{align*}
\left\{  \begin{array}{rl} i\,\partial_t u +\Delta_x u &=0\\
u(\cdot,0)& = f\in L^2(\mathbb{R}^2) \end{array}\right.
\end{align*}
and let $X$ denote the X-ray transform
$$
Xf(\ell):=\int_\ell f,$$
where $\ell$ belongs to the manifold $\mathcal{L}$ of affine lines in $\mathbb{R}^2$, endowed with the natural invariant measure (see Section \ref{sec:HD}).

In two dimensions our main result is the following:
\begin{theorem}\label{thm:main2}
\begin{equation}\label{main2}\|X(|u|^2)\|_{L^3_{t,\ell}(\mathbb{R}\times\mathcal{L})}\leq  \left(\frac{\pi}{2}\right)^{1/3}\|f\|^2_{L^2(\mathbb{R}^2)},
\end{equation}
with equality if and only if $f(x) = e^{a|x|^2 + b \cdot x + c}$ for some $\mathrm{Re}(a) < 0$, $b \in \mathbb{C}^d$ and $c\in \mathbb{C}$.
\end{theorem}
At first sight this inequality might seem rather unusual, and so some contextual remarks are in order.
The first point to make is that \eqref{main2} with a \emph{suboptimal} constant $C$ may be seen quickly by an application of the well-known (\cite{OS}, \cite{StrichartzKplane}, \cite{Cald}) X-ray transform estimate $\|Xf\|_{L^3(\mathcal{L})}\lesssim\|f\|_{L^{3/2}(\mathbb{R}^2)}$, followed by the Strichartz estimate $\|u\|_{L^6_tL^3_x(\mathbb{R}\times\mathbb{R}^2)}\lesssim\|f\|_2$; the latter following by interpolating energy conservation with the $L^2 \to L^4_{t,x}$ Strichartz inequality. Thus the main content of Theorem \ref{thm:main2} is the identification of the best constant and the characterisation of extremisers. It is perhaps curious to note that the $L^{3/2}\rightarrow L^3$ inequality for $X$ is known in sharp form, where the extremisers are quite different from gaussian (see \cite{Christ}). On the other hand, the $L^2\rightarrow L^6_tL^3_x$ Strichartz inequality is not known in sharp form, yet is (tentatively) conjectured to have only gaussian extremisers (see \cite{ChristQuil} for the related observation that gaussians are critical points for such a Strichartz inequality). Our forthcoming higher-dimensional generalisations of \eqref{main2} are rather deeper, and do not appear to permit such interpretations.

It should be pointed out that norm estimates involving $X(|u(\cdot, t)|^2)$, and more generally $R(|u(\cdot,t)|^2)$ where $R$ denotes the \emph{Radon transform} on $\mathbb{R}^d$, have arisen before in work of Planchon and Vega (\cite{PV}, \cite{VegaEsc}). In particular, when $d=2$ they observe that
\begin{equation}\label{xpv}
\|X(|u(\cdot,t)|^2)\|_{L^2_tL^2_\omega L^{2,1}_s}\lesssim\|f\|_{L^{2,1/2}}^2,
\end{equation}
where $L^{2,\sigma}$ denotes the standard $L^2$ homogeneous Sobolev space of order $\sigma$ (often denoted $\dot{H}^\sigma$) in the variables dictated by context.
Here a line $\ell\in\mathcal{L}$ is parametrised by its direction $\omega\in\mathbb{S}^1$ and its (signed) distance $s\in\mathbb{R}$ from the origin in the standard way. The similar estimate
\begin{equation}\label{xpvs}
\|X(|u(\cdot,t)|^2)\|_{L^2_tL^2_\omega L^{2,1/2}_s}\lesssim\|f\|_{L^2}^2,
\end{equation}
is also true, as can be seen by sequentially applying the $L^2\rightarrow L^2_\omega L^{2,1/2}_s$ bound on $X$ (see \cite{StrichartzKplane}) followed by the $L^4_{t,x}$ Strichartz estimate mentioned above.
However, as is pointed out in \cite{VegaEsc},
being on $L^2$, both of these estimates only involve the X-ray transform on a superficial level, as the elementary formula
$$
\|Xf\|_{L^2_\omega L^{2,\sigma+1/2}_s}=C\|f\|_{L^{2,\sigma}}
$$
reveals (both \eqref{xpv} and \eqref{xpvs} are known with sharp constant and identification of extremisers; see \cite{BBJP} for further discussion). As may be expected, non-$L^2$ inequalities, such as \eqref{main2}, are quite different in this respect. In particular, it is perhaps natural to expect a variety of (mixed-norm) estimates on $X(|u(\cdot,t)|^2)$ which \emph{cannot} be obtained by concatenating X-ray and classical Strichartz estimates, although we do not pursue this point here.

Our proof of \eqref{main2}, and its various multidimensional generalisations, will pass through a simple reduction to
certain sharp $L^2$ inequalities for Fourier extension operators, which we now describe.
Inequality \eqref{main2} is easily seen to have an equivalent trilinear formulation
\begin{align*}
\int_{\mathbb{R}}\int_{\mathcal{L}}X(|u_1(\cdot,t)|^2)(\ell)X(|u_2(\cdot,t)|^2)(\ell)&X(|u_3(\cdot,t)|^2)(\ell) \, d\ell dt \\
&\leq \frac{\pi}{2}\|f_1\|_{L^2(\mathbb{R}^2)}^2\|f_2\|_{L^2(\mathbb{R}^2)}^2\|f_3\|_{L^2(\mathbb{R}^2)}^2,
\end{align*}
and using an identity of Drury \cite{Drury84} (which, after a suitable parametrisation of $\mathcal{L}$, amounts to an application of Fubini's theorem), this becomes
\begin{align}\label{main1m}
\int_{\mathbb{R}}
\int_{(\mathbb{R}^{2})^{3}}
|u_1(x_1,t)|^2|u_2(x_2,t)|^2&|u_{3}(x_{3},t)|^2\,\delta(\rho(x))dxdt\\
&\leq \frac{\pi}{2}\|f_1\|_{L^2(\mathbb{R}^2)}^2\|f_2\|_{L^2(\mathbb{R}^2)}^2\|f_3\|_{L^2(\mathbb{R}^2)}^2,\nonumber
\end{align}
where
\begin{displaymath}
\rho(x)=\det\left(
\begin{array}{cccc}
1 & 1& 1 \\
x_1 & x_2 & x_3
\end{array}\right);
\end{displaymath}
here $x=(x_1,x_2,x_{3})\in\R^2\times\R^2\times\R^2$, and we observe that $\rho(x)=0$ if and only if $x_1,x_2,$ and $x_{3}$ are co-linear points in $\R^2$. While the methods and motivations are quite different, it seems worthwhile to note that a trilinear weighted inequality with the closely related weight $\rho^{-1}$ appeared recently in \cite{Aimpaper}.

The ostensibly multilinear inequality \eqref{main1m} is a special case of the following more general linear inequality
for the solution
$U$ of the Schr\"odinger equation in spatial dimension $6$ with initial datum $F$.
\begin{theorem}\label{thm:nontensor2}
\begin{equation}\label{nontensor}
\int_{(\mathbb{R}^2)^3\times\mathbb{R}}|U(x,t)|^2\,\delta(\rho(x))dxdt\leq \frac{\pi}{2}\|F\|_2^2,
\end{equation}
with equality when $F$ is radial.
\end{theorem}
Passing from \eqref{nontensor} to \eqref{main1m} merely requires us to take $F=f\otimes f\otimes f$, and note that this yields $U=u\otimes u\otimes u$.
However, unlike in Theorem \ref{thm:main2}, we do not claim a characterisation of the identified extremisers here. 

One may wish to interpret \eqref{nontensor} as the $t$-integrability of
$$
H(t):=\int_{(\R^2)^3}|U(x,t)|^2\,\delta(\rho(x))dx,
$$
which, informally, is a measure of the likelihood of three quantum particles in the plane being close to collinear at time $t$. As usual in the setting of Strichartz inequalities, the time integrability of a spatial expression such as $H$ is a manifestation of the dispersive nature of the Schr\"odinger equation. This physical perspective, and indeed much of this paper, is inspired by the analysis of linear and bilinear virials in \cite{PV}; see also \cite{VegaEsc}. A further source of motivation for our work may be found in the multilinear restriction theory developed in \cite{BCT} and \cite{BG}, where the quantity $\rho(x)$ and its variants play a decisive role.

As we shall see, Theorem \ref{thm:nontensor2} is equivalent to the following sharp theorem for the \textit{Fourier extension operator}
$$
\widehat{gd\sigma}(x)=\int_{\mathbb{S}^{5}}e^{-ix\cdot\omega}g(\omega)\,d\sigma(\omega),
$$
where
$d\sigma$ denotes surface measure on the unit sphere in $\mathbb{R}^{6}$.
\begin{theorem}\label{thm:extension2}
\begin{equation}\label{extn2}
\int_{\mathbb{R}^6}|\widehat{gd\sigma}(x)|^2\,\delta(\rho(x))dx\leq \frac{1}{2}(2\pi)^6 \|g\|_2^2,
\end{equation}
with equality when $g$ is a constant function.
\end{theorem}

We conclude this section with some further contextual remarks. Inequality \eqref{extn2} belongs to a much broader class of ``weighted" inequalities for the extension operator taking the form
\begin{equation}\label{SM}
\int_{\mathbb{R}^d}|\widehat{gd\sigma}(Rx)|^2\,d\nu(x)\lesssim R^{-\beta}\sup_{r>0}\left\{\frac{\nu(B_r)}{r^{\alpha}}\right\}\|g\|_2^2,
\end{equation} where $R$ is a large parameter and the supremum is taken over all  euclidean balls $B_r$ of radius $r$ in  $\R^d$. This family of inequalities has been studied extensively, with the object being to establish the best decay exponent $\beta$ as a function of the dimension parameter $\alpha$. Inequalities of this type are responsible for the best known results in the direction of the Falconer distance set problem, and have had applications in the setting of Schr\"odinger maximal estimates.
Testing \eqref{SM} on characteristic functions of small caps raises the possibility that \eqref{SM} might hold for $\alpha=\beta=d-2$ for every $d\geq 2$. This turns out to be elementary for $d=2,3$, and \textit{false} for $d\geq 5$; see \cite{IR07} and \cite{LR}. (The case $d=4$ is less clear, but the analogous inequality where the sphere is replaced by a piece of paraboloid is false; see \cite{BBCRV}.) 
Despite this failure for general measures $\nu$, our inequality \eqref{extn2} establishes that \eqref{SM} nonetheless holds (and has content) in the special case $\nu=\delta\circ\rho$ when $d=6$ and $\alpha=\beta=4$; in particular, it is straightforward to see that
$$\sup_{r>0}\left\{\frac{\nu(B_r)}{r^{4}}\right\}<\infty.$$ We remark that the homogeneity of $\delta\circ\rho$ allows the parameter $R$ in \eqref{SM} to factor out in this case, reducing it to \eqref{extn2}.
On a superficial level, \eqref{extn2} is reminiscent of elements of Guth's polynomial partitioning approach to the restriction problem \cite{GuthRest2}; see also \cite{BS}, where $L^p$ estimates for extension operators restricted to other varieties are considered.


\section{Results in higher dimensions} \label{sec:HD}

In this section we present our extension of Theorem \ref{thm:main2} to general dimensions $d$. For $d\geq 3$, it is natural to consider not just the X-ray transform, but more generally the \emph{$k$-plane transform} for any $1\leq k\leq d-1$. As we shall see, this will also involve multidimensional extensions of both Theorems \ref{thm:nontensor2} and \ref{thm:extension2}.


The statements of our results in general dimensions require us to first establish some notation.

\subsection*{Notation} 
The Grassmann manifold of all $k$-dimensional subspaces of $\mathbb{R}^d$ will be denoted by $\G{k}$. The uniform measure on $\G{k}$ is the measure which is invariant under the action of the orthogonal group; this measure is unique up to a constant. Motivated by the realisation of $\G{k}$ as a homogeneous space, we normalise $d\muG$ by
\[ \int_{\G{k}}d\muG = \frac{|\mathbb{S}^{d-1}| \cdots |\mathbb{S}^{d-k}| } { |\mathbb{S}^{k-1}| \cdots |\mathbb{S}^{0}|} \] 
where, for every $n$, $|\mathbb{S}^{n}|$ denotes the surface area of the $n$-sphere.  
Further, define 
\begin{equation}\label{gammadef}
\gamma_n(z)=\prod_{j=0}^{n-1}\frac{1}{2\pi^{(z-j)/2}}\Gamma\bigg(\frac{z-j}{2}\bigg).
\end{equation}
Then 
\[ \int_{\G{k}}d\muG = \frac{\gamma_k(k) } { \gamma_{k}(d)} \] 
The manifold of all affine $k$-dimensional subspaces of $\mathbb{R}^d$ will be denoted by $\M{k}$; this should not be confused with $\Mat(d,k)$ which we shall use to denote the space of $d$ by $k$ matrices. We write $\muM$ for the uniform measure on $\M{k}$; since all elements of $\M{k}$ can be described by an element of $\G{k}$ and a perpendicular translation, we may express $\muM$ as the product of $\muG$ and Lebesgue measure on the perpendicular $(n-k)$-plane.  The \emph{$k$-plane transform} $T_{d,k}$ is given by
\[
T_{d,k} f(\Pi) = \int_{\Pi} f
\]
for $\Pi \in \M{k}$. Evidently, $T_{d,1}$ is the X-ray transform denoted earlier by $X$, while $T_{d,d-1}$ is the Radon transform denoted earlier by $R$.

The Stiefel manifold of orthonormal $k$-frames in $\R^d$ is given by
\[
\V{k}= \{A \in \Mat(d,k): A^TA = I\}.
\]
We  write $\muV$ for the uniform measure on $\V{k}$ which we normalise by 
\[ \int_{\V{k}} d\muV = \frac{1}{\gamma_k(d)}.\]

We shall take the following version of the Fourier transform
\[
\widehat{f}(\xi) = \int_{\mathbb{R}^n} f(x)e^{-ix\cdot \xi} \, dx
\]
for appropriate functions on $\mathbb{R}^n$ (where $n = d$ or $n = d(d+1)$). These conventions allow us to explicitly display various optimal constants which appear below. In order to display these constants succinctly, it is helpful to introduce the notation
\[
\mathbf{C}_{d.k} = (2\pi)^{d(k+1)} (d+1)^{\frac{d(d-k)+k-3}{2}} |\mathbb{S}^{d(d-k)-1}|
\]
and
\[
\mathbf{D}_{d,k} = \frac{1}{\gamma_{k}(k) \gamma_{d-k}(d-k)}.
\]


Finally, for $\xi = (\xi_1,\ldots,\xi_{d+1}) \in\R^{d(d+1)}$ with each $\xi_i\in\R^d$, we shall use $\var_\xi$ to denote the covariance matrix of the vector-valued random variable taking the values $\{\xi_i\}_{i=1}^{d+1}$ with equal probability. Thus the $j$th diagonal element of $\var_{\xi}$ is the variance of the $j$th coordinate of this random variable, and $(\var_{\xi})_{j,\ell}$ is the covariance between the $j$th and $\ell$th coordinates. We may write this covariance matrix as
\[
\var_\xi = \frac{1}{2(d+1)^2}  \sum_{i,j=1}^{d+1} (\xi_i-\xi_j)(\xi_i-\xi_j)^T.
\]

\subsection*{Statements of results in general dimensions}
It is natural to try to generalise \eqref{main2} by taking the $L^{d+1}_{t,\Pi}(\mathbb{R}\times\M{k})$ norm of $T_{d,k}(|u(\cdot,t)|^2)$, where $u:\mathbb{R}^d\times\mathbb{R}\rightarrow\mathbb{C}$ is a solution to the $d$-dimensional Schr\"odinger equation. More generally one might consider the $L^{1}_{t,\Pi}(\mathbb{R}\times\M{k})$ norm of $T_{d,k}(|u_1(\cdot,t)|^2)\cdots T_{d,k}(|u_{d+1}(\cdot,t)|^2)$, for solutions $u_1,\hdots,u_{d+1}$. While this distinction is entirely superficial when $d=2$, in higher dimensions it acquires some content as certain geometric weight factors emerge in the resulting sharp bounds which encode nontrivial $d(d+1)$-dimensional structure. For example, for the Radon transform in $\R^4$ we shall establish the sharp inequality
\[
\int_{\R}\int_{\mathcal{M}_{4,3}}\prod_{i=1}^{5}R( |u_i(\cdot,t)|^2)(\Pi)\,d\muM(\Pi)dt\leq \frac{\pi  }{2^4(2\pi)^{20} } \int_{(\R^{4})^5} \prod_{i=1}^{5}|\widehat{f}_i(\xi_i)|^2\bigg( \sum_{i,j=1}^{5}  |\xi_i-\xi_j|^2\bigg)\,d\xi,
\]
with equality if and only if $f_1,\hdots,f_5$ are scalar multiples of the same gaussian function.

The weight $\sum_{i,j=1}^{d+1}  |\xi_i-\xi_j|^2$ appearing above has a very simple probabilistic interpretation; up to a dimensional constant, it is simply the trace of the covariance matrix $\var_\xi$. More precisely, one can check that
\begin{equation}\label{simplestweight}
 \tr\var_\xi = \frac{1}{2(d+1)^2}  \sum_{i,j=1}^{d+1} |\xi_i-\xi_j|^2.
\end{equation}
Powers of $(\tr\var_\xi)$ have appeared previously as weights in related sharp inequalities for solutions to the Schr\"odinger equation; see Carneiro \cite{Carneiro}.

Typically, the controlling weight functions may be rather more complicated than this simple variance expression, however a clear statistical flavour is retained. For example, in either the case of the $3$-dimensional $X$-ray transform or the $6$-dimensional Radon transform, the resulting weight turns out to be a constant multiple of
\begin{equation}\label{6dRadonweight}
{ (\tr\var_\xi)^2        +2\tr\var_\xi^2 =  \frac{1}{4(d+1)^4}\sum_{i,j,k,l=1}^{d+1} |\xi_i-\xi_j|^2 |\xi_k-\xi_l|^2  +2|\langle\xi_i-\xi_j,\xi_k-\xi_l\rangle|^2.}
\end{equation}
For a general pair ${d,k}$ such closed form expressions hide an interesting geometric structure and the controlling weight is defined as follows. Denoting by $P_\Pi\in \Mat(d,d)$ the orthogonal projection onto $\Pi\in \M{d-k}$, we define
\begin{equation}\label{weightdef}
I_{d,k}(\xi) := \int_{\G{d-k}}\left( \tr (P_\Pi \var_\xi) \right)^\frac{d(d-k)-2}{2}
\,d\muG.
\end{equation}
An important feature of $I_{d,k}$ is that it is homogeneous of degree $d(d-k)-2$. It may also be of interest to note that $\tr(P_\Pi\var_\xi)$ is equal to the trace of the covariance matrix of the vector-valued random variable taking the values $\{P_\Pi\xi_i\}_{i=1}^{d+1}$ with equal probability. Therefore, 
\[ 
I_{d,k}(\xi) = \left(2(d+1)^2\right)^{\frac{2-d(d-k)}{2}} \int_{\G{d-k}}\bigg(\sum_{i,j=1}^{d+1}|P_\Pi(\xi_i-\xi_j)|^2 \bigg)^\frac{d(d-k)-2}{2}\,d\muG.
\]

Although these weights differ from powers of the simplest variance expression \eqref{simplestweight}, we shall see that they are in fact comparable; see Proposition \ref{prop:carneirocomp}. While these more sophisticated weights may have some independent interest, the main purpose for identifying them is to obtain sharp estimates and characterise the cases of equality, as follows.

\begin{theorem} \label{thm:main}
Suppose $d \geq 2$ and $1 \leq k \leq d-1$. Then
 \begin{equation}\label{fulltensor}
\int_{\mathbb{R}} \int_{\mathcal{M}_{d,k}} \prod_{i=1}^{d+1} T_{d,k}(|u_i(\cdot,t)|^2)(\Pi)\, d\muM(\Pi)dt \leq \frac{\pi \mathbf{C}_{d,k}}{(2\pi)^{2d(d+1)}} \int_{(\mathbb{R}^d)^{d+1}} \prod_{i=1}^{d+1} |\widehat{f_i}(\xi_i)|^2I_{d,k}(\xi) \, d\xi,
\end{equation}
where the constant is optimal and is attained if and only if $f_i(x) = e^{a|x|^2 + b \cdot x + c_i}$ for some $\mathrm{Re}(a) < 0$, $b \in \mathbb{C}^d$ and $c_i \in \mathbb{C}$.
\end{theorem}

Of course Theorem \ref{thm:main} reduces to Theorem \ref{thm:main2} on setting $(d,k)=(2,1)$. Some further context for this theorem is available on formally taking $k=0$, interpreting $T_{d,0}$ as the identity operator, and observing that
$
I_{d,0}(\xi)= (\tr \var_\xi)^{(d^2-2)/2}$.
In the case $d=1$, inequality \eqref{fulltensor} thus becomes the well-known (and elementary) bilinear Strichartz estimate
$$
\int_{\mathbb{R}}\int_{\mathbb{R}}|u_1(x,t)|^2|u_2(x,t)|^2\,dxdt\leq C \int_{\mathbb{R}}\int_{\mathbb{R}}\frac{|\widehat{f}_1(\xi_1)|^2|\widehat{f}_2(\xi_2)|^2}{|\xi_1-\xi_2|}\,d\xi_1d\xi_2,
$$
which is in fact an identity for initial data with disjoint Fourier supports. On formally taking $k=0$ in higher dimensions, Theorem \ref{thm:main} reduces to a sharp inequality of Carneiro \cite{Carneiro}.

Theorem \ref{thm:main} will follow from a similar statement for a general solution $U$ to the Schr\"odinger equation in $\mathbb{R}^{d(d+1)}$, as with Theorem \ref{thm:main2}. In order to formulate this linear statement we shall need to identify the appropriate singular distributions which generalise $\delta \circ \rho$.
To this end, consider the analytic family of distributions:
\[
A_s =\frac{1}{\gamma_d(s)}|\rho|^{-d+s},
\]
where
\begin{equation} \label{e:rhodefn}
\rho(x)=\det\left(
\begin{array}{cccc}
1&  1 &\cdots & 1\\
x_1 & x_2 &
\cdots & x_{d+1}\\
\end{array}\right).
\end{equation}
The function $\gamma_d(s)$ defined in \eqref{gammadef} is chosen to cancel the poles of the function $|\rho|^{-d+s}$. 

The connection between the two-dimensional $X$-ray transform and $\delta \circ \rho$ described in the previous section extends to higher dimensions through the identity
\begin{equation}\label{DrurysID}
\langle A_k, f_1\otimes\cdots\otimes f_{d+1} \rangle = \mathbf{D}_{d,k} \int_{\mathcal{M}_{d,k}} \prod_{i=1}^{d+1}T_{d,k} f_i(\Pi)\,d\muM(\Pi).
\end{equation}
Drury proves this formula in \cite{Drury84} without identifying the constant. We shall give a short proof of \eqref{DrurysID} using a formula of Rubin \cite{Rubin06} in Section \ref{sec:prelim}.

The resulting extension of Theorem \ref{thm:nontensor2} is the following.
\begin{theorem}\label{thm:nontensor}
Suppose $d \geq 2$ and $1 \leq k \leq d-1$. Then
\begin{equation}\label{nontensord}
\int_{\mathbb{R}} \langle A_k,|U(\cdot,t)|^2\rangle \,dt\leq \frac{\pi \mathbf{C}_{d,k} \mathbf{D}_{d,k} }{(2\pi)^{2d(d+1)}}\int_{\R^{d(d+1)}}|\widehat{F}(\xi)|^2 I_{d,k}(\xi) \,d\xi,
\end{equation}
where the constant is optimal and is attained if $F$ is radial.
\end{theorem}
Again mimicking the line of reasoning in two dimensions, we have that Theorem \ref{thm:nontensor} is equivalent to the following statement in terms of the extension operators.
\begin{theorem}\label{thm:extension}
Suppose $d \geq 2$ and $1 \leq k \leq d-1$. Then
\begin{equation}\label{extn}
\langle A_k,|\widehat{gd\sigma}|^2\rangle \leq \mathbf{C}_{d,k} \mathbf{D}_{d,k} \int_{\mathbb{S}^{d(d+1)-1}}|g(\omega)|^2I_{d,k}(\omega) \, d\sigma(\omega),
\end{equation}
where the constant is optimal and is attained if $g$ is a constant function.
\end{theorem}

\begin{remark}
 Theorem \ref{thm:extension} also leads to sharp estimates as in Theorem \ref{thm:nontensor}, for the extension operator associated with more general surfaces produced by revolution of a curve. 
 In particular, the analog of Theorem \ref{thm:nontensor} holds for the extension operator associated to the surface $S$ whenever there exists a monotone function $\Sigma$ such that 
 $S=\{\left(y,\Sigma(|y|)\right):\;y\in \R^{d(d+1)}\}$ in $\R^{d(d+1)}$. For example, if $S$ is the light cone in $\R^{d(d+1)+1}$ corresponding to $\Sigma(r) = r$, the above inequality gives
$$
\int_{\mathbb{R}} \langle A_k,|U(\cdot,t)|^2\rangle \,dt\leq \frac{\mathbf{C}_{d,k} \mathbf{D}_{d,k}}{(2\pi)^{2d(d+1)-1}}  \int_{\R^{d(d+1)}}|\widehat{F}(\xi)|^2I_{d,k}(\xi)|\xi| \,d\xi,
$$
where $U$ is the solution to the (one-sided) wave equation with initial data $F:\R^{d(d+1)}\rightarrow \C$. However, the tensor product of solutions to a low-dimensional wave equation is not a solution to a high-dimensional wave equation. For this reason, lower dimensional estimates for the wave equation cannot be derived from the above inequality.
\end{remark}

We end this section with some further structural remarks on the weight $I_{d,k}$ defined by \eqref{weightdef}.

First we establish comparability between our weights and the simpler family $\tr{\var_\xi}$ raised to the appropriate power. 
\begin{proposition}\label{prop:carneirocomp}
Suppose $d \geq 2$ and $1 \leq k \leq d-1$. Then there exists constants $0 < c_{d,k} \leq C_{d,k} < \infty$ such that
\[
c_{d,k} (\tr\var_\xi)^{\frac{d(d-k)-2}{2}} \leq I_{d,k}(\xi) \leq C_{d,k} (\tr\var_\xi)^{\frac{d(d-k)-2}{2}}
\]
for all $\xi\in\mathbb{R}^{d(d+1)}$, where $\tr\var_\xi$ is given by \eqref{simplestweight}.
\end{proposition}

Secondly, we note that when $d(d-k)$ is even the average over $\G{d-k}$ can be evaluated to produce closed-form expressions for $I_{d,k}$ such as \eqref{6dRadonweight}. However, in higher dimensions the resulting constants are quite complicated.

\begin{proposition}\label{prop:evenmore}
Suppose $d \geq 2$ and $1 \leq k \leq d-1$. If either $d$ is even or $k$ is odd, then $I_{d,k}(\xi)$ is a linear combination of terms of the form
\[
(\tr \var_\xi^{\alpha_1})^{\beta_1} (\tr \var_\xi^{\alpha_2} )^{\beta_2}\cdots(\tr \var_\xi^{\alpha_\kappa})^{\beta_\kappa}
\]
where $\alpha_1\beta_1 +\alpha_2\beta_2+\cdots+ \alpha_\kappa\beta_\kappa = \frac{d(d-k)-2}{2}$ and each of the $\alpha_i$ and $\beta_i$ are nonnegative integers.
\end{proposition}


\section{Preliminaries regarding the family of distributions $A_s$} \label{sec:prelim}
We will understand the analytic family of distributions $A_s$ defined in Section \ref{sec:HD} in terms of the more standard family of distributions
\[ \zeta_s=\frac{1}{\gamma_d(s)}|\det(X)|^{-d+s} \]
where as before $\gamma_d(s)$ is the normalised product of gamma functions \eqref{gammadef}.  Here $X\in\Mat(d,d)$ which we implicitly identify throughout with $\R^{d^2}$. In particular we take the inner product of $X,Y\in\Mat(d,d)$ to be the Frobenius inner product. 

For all $s\in\C$, it is known that $\zeta_s$ is a well-defined tempered distribution, understood in the sense of analytic continuation, and
\begin{equation}\label{Zhat}
\widehat{\zeta_s}=\frac{(2\pi)^{ds}}{\gamma_d(d-s)}|\det(X)|^{-s} =(2\pi)^{ds}\zeta_{d-s}
\end{equation}
(see Rubin \cite{Rubin06} and the references therein). 

Rubin (\cite{Rubin06}, Corollary 4.5) additionally proves the identity
\begin{equation}\label{RubinID}
\langle \zeta_k, f\rangle =\frac{1}{\gamma_{d-k}(d-k)} \int_{\V{k}} \int_{\Mat(k,d)} f(VY) \,dY d\muV,
\end{equation}
for the singular cases $1 \leq k \leq d$.

To relate this to $A_s$, we let $L$ be the linear map
\[
L(x_1,\ldots,x_{d+1})= (x_1 -x_{d+1}, \ldots, x_d - x_{d+1}, x_{d+1} ).
\]
Recalling the definition of $\rho$ in \eqref{e:rhodefn}, it follows that
\[
\rho(x_1, \ldots,x_{d+1})= \det\circ L(x_1,\ldots,x_{d+1}),
\]
and hence
\begin{equation}\label{Anew} A_s  = (\zeta_s\otimes1)\circ L.
\end{equation}

\begin{lemma}\label{lem:Ahat} For all $s\in\C$, $A_s$ is a well-defined tempered distribution and
\begin{equation}\label{Ahat}
\widehat{ A_s }  = \frac{(2\pi)^{ds}}{\gamma_d(d-s)} |\det(\xi_1,\ldots,\xi_d)|^{-s}\delta\left(\sum_{i=1}^{d+1}\xi_i\right) = (2\pi)^{ds}\zeta_{d-s}\cdot\delta\left(\sum_{i=1}^{d+1}\xi_i\right).
\end{equation}
In particular, $\widehat{A_k}$ is positive whenever $1 \leq k \leq d-1$ is an integer.
\end{lemma}

\begin{proof} Since $\det{L}=1$, it follows from \eqref{Anew} that
\[
\widehat{ A_s }  = \widehat{(\zeta_s\otimes1)}\circ L^{-T}=(\widehat{\zeta_s}\otimes\delta)\circ L^{-T}
\]
where
\[
L^{-T}(x_1,\ldots,x_{d+1})= \bigg(x_1 , \ldots, x_d ,\sum_{i=1}^{d+1}x_i \bigg).
\]
Therefore, by \eqref{Zhat},
\[
\widehat{ A_s }  = \frac{(2\pi)^{d(s+1)}}{\gamma_d(d-s)} |\det(\xi_1,\ldots,\xi_d)|^{-s}\delta\left(\sum_{i=1}^{d+1}\xi_i\right) = (2\pi)^{d(s+1)}\zeta_{d-s}\cdot\delta\left(\sum_{i=1}^{d+1}\xi_i\right).
\]
Positivity of $\widehat{A_k}$ for integers $1 \leq k \leq d-1$ can now be seen via Rubin's formula \eqref{RubinID}.
\end{proof}

We end this section by using the characterisation of $A_s$ in  \eqref{Anew} and Rubin's formula  \eqref{RubinID} to give a short proof of Drury's formula \eqref{DrurysID}.

\begin{proof}[Proof of \eqref{DrurysID}]  By \eqref{Anew} and \eqref{RubinID},
\begin{align*}
& \langle A_k, f_1\otimes\cdots\otimes f_{d+1} \rangle \\
& = \langle \zeta_k\otimes1,( f_1\otimes\cdots\otimes f_{d+1})\circ L^{-1}\rangle \\
& =  \frac{1}{\gamma_{d-k}(d-k)} \int_{\V{k}} \int_{\Mat(k,d)} \prod_{i=1}^df_i\bigg(\sum_{j=1}^k v_jy_{ji}+x_{d+1}\bigg)f_{d+1}\left(x_{d+1}\right) \,dx_{d+1}dYd\muV.
 \end{align*}
Here $v_j\in\R^d$ denotes the $j$-th column of $V$ and $y_{ji}$ is the $(j,i)$-th entry of $Y \in \Mat(k,d)$.  Writing $x_{d+1}$  in terms of orthonormal vectors $\{v_j\}_{j=1}^k$, there exist  $\{y_{j,d+1}\}_{j=1}^k$, and $x'$ satisfying $\langle x', v_j \rangle =0$ for each $1 \leq j \leq k$ such that 
\[
x_{d+1} = \sum_{j=1}^k v_jy_{j(d+1)}+x'.
\]
Let $\Pi$ denote the $k$-plane spanned by the $\{v_j\}_{j=1}^k$, thus $x\in\Pi^{\perp}$.
Making the further change variables $y_{ji}+y_{j,k+1}\mapsto y_{ji}$, yields
\[
\langle A_k, f_1\otimes\cdots\otimes f_{d+1} \rangle= \frac{1}{\gamma_{d-k}(d-k)} \int_{\V{k}} \int_{\Mat(k,d+1)}
\prod_{i=1}^{d+1}f_i\bigg(\sum_{j=1}^kv_jy_{ji}+x'\bigg) \,d_{\Pi^{\perp}}(x') dYd\muV.
\]
Observe that the families $\{v_j\}_{j=1}^k$ range over all orthonormal bases of all $k$-planes through origin and that $x'$, which depends on $\{v_j\}_{j=1}^k$, is a translation perpendicular to this $k$-plane. From this we may deduce \eqref{DrurysID}, noting that the factor of $(\gamma_k(k))^{-1}$ results from the change of variables from $\muV$ to $\muG$. 
\end{proof}


\section{Proof of Theorems \ref{thm:nontensor} and \ref{thm:extension}  } \label{sec:mainproof}

\subsection{Proof of Theorem \ref{thm:extension}}
By expanding the square, we have
\begin{equation*}
\langle A_k,|\widehat{gd\sigma}|^2\rangle =  \int_{\mathbb{S}^{d(d+1)-1}}\int_{\mathbb{S}^{d(d+1)-1}}g(\omega)\overline{g(\omega')}\widehat{A_k}(\omega-\omega') \, d\sigma(\omega)d\sigma(\omega')
\end{equation*}
and using the fact that $\widehat{A}_s(-\xi)=\widehat{A}_s(\xi)$ (which can be seen quickly from Lemma \ref{lem:Ahat}) we can write
\begin{align*}
\langle A_k,|\widehat{gd\sigma}|^2\rangle & = \int_{\mathbb{S}^{d(d+1)-1}}|g(\omega)|^2\int_{\mathbb{S}^{d(d+1)-1}}\widehat{A_k}(\omega-\omega') \,d\sigma(\omega')d\sigma(\omega) \\
&\qquad \qquad -\frac{1}{2}\int_{\mathbb{S}^{d(d+1)-1}}\int_{\mathbb{S}^{d(d+1)-1}}|g(\omega)-g(\omega')|^2\widehat{A_k}(\omega-\omega') \, d\sigma(\omega)d\sigma(\omega').
\end{align*}
Since $\widehat{A_k}$ is nonnegative (Lemma \ref{lem:Ahat}), we immediately obtain the sharp inequality
\[
\langle A_k,|\widehat{gd\sigma}|^2\rangle \leq \int_{\mathbb{S}^{d(d+1)-1}}|g(\omega)|^2\widehat{A_k} * d\sigma(\omega) \, d\sigma(\omega)
\]
with equality when $g$ is a constant function.

In order to complete the proof of Theorem \ref{thm:extension}, it remains to prove the following.
\begin{proposition}\label{prop:weightchar}
Suppose $d \geq 2$ and $1 \leq k \leq d-1$. Then
\begin{equation} \label{e:weightchar}
\widehat{A_k}*d\sigma (\omega) = \mathbf{C}_{d,k}\mathbf{D}_{d,k} \int_{\G{d-k}}\left( \tr (P_\Pi \var_\omega) \right)^\frac{d(d-k)-2}{2} d\mu_{\mathcal{G}}
\end{equation}
for all $\omega\in\mathbb{S}^{d(d+1)-1}$.
\end{proposition}

\begin{proof}
Firstly, we observe that $\widehat{A_k}*d\sigma$ is a well-defined distribution  as $d\sigma$ is compactly supported and $\widehat{A_k}$ is well-defined by Lemma \ref{lem:Ahat}. In order to obtain the claimed identity in Proposition \ref{prop:weightchar}, we use Lemma \ref{lem:Ahat} to write
\[
\widehat{A_k}*d\sigma (\omega) = \frac{2(2\pi)^{d(k+1)}}{\gamma_d(d-k)}\int_{\R^{d(d+1)}} \delta(1-|\omega-\eta|^2)|\det(\eta_1,\ldots,\eta_d)|^{-k}\delta\bigg(\sum_{i=1}^{d+1}\eta_i\bigg) \, d\eta,
\]
and hence
\begin{align*}
\widehat{A_k}*d\sigma (\omega) & = \frac{(2\pi)^{d(k+1)}}{\gamma_d(d-k)}\int_{(\R^d)^d} \delta\bigg(\sum_{i=1}^d |\eta_i|^2+\sum_{i<j}\langle\eta_i,\eta_j\rangle  - \sum_{i=1}^{d}\langle \omega_i- \omega_{d+1},\eta_i \rangle\bigg)|\det N|^{-k}\,d\eta_1\cdots d\eta_d,
\end{align*}
where $N \in \Mat(d,d)$ has $i$-th column equal to $\eta_i$.

Next we claim that if
\[
M := I + (\sqrt{d+1}-1){\bf 11}^T
\]
with ${\bf 1}$ denoting the unit vector in the direction $(1,1,\ldots,1)$, then
\begin{equation}\label{COV}
	 \frac{1}{2} \sum_{i=1}^d|\xi_i|^2 = \sum_{i=1}^d |\eta_i|^2 +\sum_{i<j}\langle\eta_i,\eta_j\rangle
\end{equation}
holds, where $X = NM$ and $X \in \Mat(d,d)$ denotes the matrix whose $i$-th column is $\xi_i$. To see this, observe that $M^2 = I +d{\bf 11}^T$, the matrix with each diagonal element equal to 2 and every other element equal to 1.
%
We have $\xi_i=\sum_{j=1}^d M_{ij} \eta_j$ and thus, by construction,
\begin{align*}
\sum_{i=1}^d |\xi_i|^2 = \sum_{j=1}^d \bigg( \sum_{i=1}^d M_{ij}^2\bigg) |\eta_j|^2 + 2\sum_{j < k}^d \bigg(\sum_{i=1}^d M_{ij}M_{ik}\bigg) \langle\eta_j,\eta_k\rangle
= \sum_{j=1}^d 2|\eta_j|^2 + 2\sum_{j<k}^d \langle\eta_j,\eta_k\rangle.
\end{align*}

Computing the eigenvalues of $M$ shows that $\det M = \sqrt{d+1}$, and therefore the change of variables $X = NM$ yields
\begin{align*}
\widehat{A_k}*d\sigma (\omega) = \frac{ (d+1)^\frac{k-1}{2}(2\pi)^{d(k+1)} }{\gamma_d(d-k)}\int_{(\R^d)^d}\delta(\tfrac{1}{2}|X|^2 - \langle \Omega' M^{-1},X\rangle) |\det(X)|^{-k} \,dX.
\end{align*}
Here, $\Omega' \in \Mat(d,d)$ has $i$-th column equal to $\omega_i- \omega_{d+1}$, and the inner product of two matrices is the Frobenius inner product, as we freely identify $\Mat(d,d)$ with $(\R^d)^d$ here and in what follows. 

In what follows we set $\Omega_M = \Omega' M^{-1}$. Switching to polar coordinates in $\R^{d^2}$,
\begin{align*}
\widehat{A_k}*d\sigma (\omega) & = \frac{2(d+1)^{\frac{k-1}{2}}(2\pi)^{d(k+1)}}{\gamma_d(d-k)}  \int_{\mathbb{S}^{d^2-1}}\int_{0}^\infty \delta((r - \langle \Theta,\Omega_M\rangle)^2 - \langle \Theta,\Omega_M \rangle^2) \frac{r^{d^2-kd-1}}{|\det \Theta|^{k}} \,dr  d\sigma (\Theta) \\
& = \frac{(d+1)^{\frac{k-1}{2}} 2^{d(d-k)-2}(2\pi)^{d(k+1)} }{\gamma_d(d-k)}\int_{\mathbb{S}^{d^2-1}} |\langle \Theta, \Omega_M\rangle|^{d(d-k)-2} \, |\det \Theta|^{-k} \, d\sigma(\Theta).
\end{align*}
Next we apply \eqref{RubinID} which yields
\[
\widehat{A_k}*d\sigma (\omega) = C_{d,k}  \int_{\V{d-k}} \int_{\Mat(d-k,d)}|\langle VY,\Omega_M\rangle|^{d(d-k)-2}\delta(|Y|^2-1) \,dYd\muV,
\]
where
\[
C_{d,k} = (d+1)^{\frac{k-1}{2}} 2^{d(d-k)-1} (2\pi)^{d(k+1)} \left(\frac{1}{\gamma_{k}(k)} \right)^2.
\]
Here we used that $|VY|^2 = |Y|^2$ for all $V \in \V{d-k}$. By a further polar coordinates change of variables and using the explicit integral
\[
\int_{\mathbb{S}^{n-1}} |\langle \Theta, V \rangle|^p \,d\sigma(\Theta) = 2\pi^{\frac{n-1}{2}} \frac{\Gamma(\frac{p+1}{2})}{\Gamma(\frac{p+n}{2})} |V|^p
\]
valid for $n \geq 2$ and $p > -n$, we obtain
\[
\widehat{A_k}*d\sigma (\omega) = (d+1)^{\frac{k-1}{2}}  (2\pi)^{d(k+1)} \frac{|\mathbb{S}^{d(d-k)-1}|}{\gamma_{k}(k)} \int_{\V{d-k}} |V^T\Omega_M|^{d(d-k)-2}\ \, d\muV.
\]
Now,
\[
|V^T\Omega_M|^2 = \tr(V^T\Omega_M\Omega_M^TV) = \tr (VV^T\Omega_M\Omega_M^T)
\]
and we observe that the matrix $VV^T$ is the orthogonal projection onto the $(d-k)$-plane, $\Pi$, spanned by the columns of $V$.   In particular, $VV^T$ depends only on this $(d-k)$-plane. By defining $P_\Pi = VV^T$ we have
\begin{equation}\label{normform}
\int_{\V{d-k}}|V^T\Omega_M|^{d(d-k)-2} \,d \muV = \int_{\V{d-k}} ( \tr (P_\Pi \Omega_M\Omega_M^T))^\frac{d(d-k)-2}{2} \,
d\muV.
\end{equation}
As the integrand depends only $\Pi$,  we have
\[
\int_{\V{d-k}}|V^T\Omega_M|^{d(d-k)-2}\,d\muV=\frac{1}{\gamma_{d-k}(d-k)}\int_{\G{d-k}} (\tr (P_\Pi \Omega_M\Omega_M^T))^\frac{d(d-k)-2}{2} \,
d\mu_\mathcal{G}.
\]
Finally, using the forthcoming Lemma \ref{lem:covar} which says that $\Omega_M\Omega_M^T=(d+1) \var_\omega $, we obtain \eqref{e:weightchar}.
\end{proof}

\subsection{Proof of Theorem \ref{thm:nontensor}} We use an argument of Vega (see, for example, \cite{VegaPAMS1988}) to show that Theorem \ref{thm:extension} implies Theorem \ref{thm:nontensor}. Despite this argument being well-known, since the statements of these theorems contain explicit constants, we provide some details. First, by simple changes of variables, we may write
\[
U(x,t) = \frac{1}{2(2\pi)^{d(d+1)}} \int_0^\infty e^{-itr} \widehat{g_rd\sigma}(\sqrt{r}x) r^{\frac{d(d+1)-2}{2}} \, dr
\]
where $g_r(\omega) := \widehat{F}(\sqrt{r}\omega)$. We now apply Plancherel's theorem in the $t$ variable, along with the fact that $A_k$ is homogeneous of degree $d(k-d)$, to obtain
\[
\int_{\mathbb{R}} \langle A_k,|U(\cdot,t)|^2\rangle \,dt = \frac{\pi}{(2\pi)^{2d(d+1)}} \int_0^\infty \langle A_k,|\widehat{g_{r^2}d\sigma}|^2\rangle r^{d(d-k)-2} \, r^{d(d+1)-1} \, dr.
\]
Applying Theorem \ref{thm:extension}, along with the fact that $I_{d,k}$ is homogeneous of degree $d(d-k)-2$ and reversing the polar coordinates change of variables, we deduce the claimed inequality in Theorem \ref{thm:nontensor}. Also, it is clear that if $F$ is radial then each $g_r$ is a constant function. Therefore, the fact that constant functions give equality in Theorem \ref{thm:extension} implies that radial functions give equality in Theorem \ref{thm:nontensor}. \qed

We end this section by proving Propositions \ref{prop:carneirocomp} and \ref{prop:evenmore} concerning the structure of the weight $I_{d,k}$.
\begin{proof}[Proof of Proposition \ref{prop:carneirocomp}]
First, we observe that when $d=2$ we only consider $k=1$ in which case $I_{2,1}$ is identically equal to one, and therefore there is nothing to prove. For $d \geq 3$, we note that
$\var_\xi$ is positive semidefinite and therefore $\var_\xi^{1/2}$ is well-defined for every $\xi\in\R^{d^2}$. Using this and \eqref{normform} (along with homogeneity) yields,
\[
I_{d,k}(\xi) = c_{d,k}\int_{\V{d-k}} |V^T\var_\xi^{1/2}|^{d(d-k)-2} \,d\mu_\mathcal{V}
\]
for some constant $c_{d,k}$. The map
\[
M \mapsto \bigg(\int_{\V{d-k}} |V^TM|^{d(d-k)-2} \, d\mu_\mathcal{V}\bigg)^{\frac{1}{d(d-k)-2}}
\]
is a norm on $\R^{d^2}$ since $d(d-k)-2 \geq 1$, and since all norms on finite dimensional vector spaces are equivalent, we immediately obtain the desired conclusion in Proposition \ref{prop:carneirocomp} by comparison with the Hilbert--Schmidt norm.
\end{proof}

\begin{proof}[Proof of Proposition \ref{prop:evenmore}]
Let $\{\lambda_1,\ldots, \lambda_d\}$ denote the eigenvalues of $\var_\xi^{1/2}$. Then
\[
I_{d,k}(\xi)= c_{d,k}\int_{\V{d-k}}\bigg(\sum_{i=1}^{d-k}\sum_{j=1}^d \lambda_j^2v_{i,j}^2\bigg)^\frac{d(d-k)-2}{2}  \,d\muV
\]
for some constant $c_{d,k}$. Since either $d$ is even or $k$ is odd it follows that $d(d-k)$ is even and therefore
\[
\bigg(\sum_{i=1}^{d-k}\sum_{j=1}^d \lambda_j^2v_{i,j}^2\bigg)^\frac{d(d-k)-2}{2} = \sum_{2(\alpha_{1,1}+\cdots+\alpha_{d-k,d})=\ell} c_{\ell,\alpha} \prod_{i=1}^{d-k}\prod_{j=1}^{d}(\lambda_jv_{i,j})^{2\alpha_{i,j}}
\]
where $\ell = d(d-k)-2$. Therefore,
\[
I_{d,k}(\xi)= c_{d,k} \sum_{2(\alpha_{1,1}+\cdots+\alpha_{d-k,d})=\ell}\bigg(c_{\ell,\alpha}\int_{\V{d-k}}\prod_{i=1}^{d-k}\prod_{j=1}^{d}(v_{i,j})^{2\alpha_{i,j}} \,d\muV\bigg)\prod_{j=1}^{d}\lambda_j^{2\sum_{i=1}^{d-k}\alpha_{i,j}}.
\]
The advantages of this expression is that it is a symmetric polynomial in the $\lambda_j^2$ (that is, it is invariant under reordering of the eigenvalues). Since symmetric polynomials are generated by the power-sum polynomials, then
\[
\sum_{j=1}^d \lambda_j^{2\beta}=\tr((\var_\xi)^\beta)
\]
and we deduce that $I_{d,k}(\xi)$ is a linear combination of terms of the form
\[
(\tr \var_\xi^{\alpha_1})^{\beta_1} (\tr \var_\xi^{\alpha_2})^{\beta_2}\cdots(\tr \var_\xi^{\alpha_\kappa})^{\beta_\kappa}
\]
where $\alpha_1\beta_1 +\alpha_2\beta_2+\cdots+ \alpha_\kappa\beta_\kappa = \ell/2 $ and each of the $\alpha_i$ and $\beta_i$ are nonnegative integers.
\end{proof}

\section{Proof of Theorem \ref{thm:main}}\label{sec:unique}

The sharp inequality in Theorem \ref{thm:main} is a consequence of Theorem \ref{thm:nontensor}; however, in order to obtain the claimed characterisation of cases of equality, it is necessary to enter into the argument which gives this implication.

\begin{proof}[Proof of Theorem \ref{thm:main}]

Arguing as in the proof of Theorem \ref{thm:extension}, we have
\begin{align*}
& \frac{(2\pi)^{2d(d+1)}}{\pi} \int_{\mathbb{R}} \int_{\mathcal{M}_{d,k}}  T_{d,k}(|U(\cdot,t)|^2)(\Pi)\, d\mu_{\mathcal{M}}(\Pi)dt \\
& = \mathbf{C}_{d,k} \int_{(\mathbb{R}^d)^{d+1}} |\widehat{F}(\xi)|^2 I_{d,k}(\xi) \, d\xi  - \frac{1}{\mathbf{D}_{d,k}} \int_{(\mathbb{R}^d)^{d+1}}  |\widehat{F}(\xi) - \widehat{F}(\eta)|^2 \, d\Sigma_\xi(\eta)d\xi,
\end{align*}
where $F = f_1 \otimes \cdots \otimes f_{d+1}$ and
\[
d\Sigma_\xi(\eta) = \widehat{A_k}(\xi-\eta)\delta(|\xi|^2 - |\eta|^2) \, d\eta.
\]

This yields the claimed estimate with optimal constant, and equality holds if and only if $\widehat{F}(\xi) = \widehat{F}(\eta)$ for almost every $\xi$ and $\eta$ in the support of the measure $d\Sigma_\xi(\eta)d\xi$. By Lemma \ref{lem:Ahat} the support of this measure, which we denote by $\supp$, is characterised by those $\xi$ and $\eta$ for which
\begin{equation*}
|\xi| = |\eta|, \qquad \sum_{i=1}^{d+1} \xi_i = \sum_{i=1}^{d+1} \eta_i,
\end{equation*}
and the points $\xi_i - \eta_i$, for $1 \leq i \leq d+1$, are co-$k$-planar. Recalling that $F = f_1 \otimes \cdots \otimes f_{d+1}$, the characterisation of extremisers follows once we show that
\[
\prod_{i=1}^{d+1} \widehat{f_i}(\xi_i) = \prod_{i=1}^{d+1} \widehat{f_i}(\eta_i)
\]
for almost every $(\xi,\eta) \in \supp$ implies that $\widehat{f_i}(\xi) = e^{a|\xi|^2 + b \cdot \xi + c_i}$, where $\mathrm{Re}(a) < 0$, $b \in \mathbb{C}^d$ and $c_1,c_2 \in \mathbb{C}$. Indeed, it is clear that this class is invariant under the Fourier transform, and such $f_i$ satisfy the above functional equation.

First, we argue that the $f_i$ are mutually parallel. For this, for each $i$, fix $\xi_i^*$ such that $\widehat{f_i}(\xi_i^*) \neq 0$ and note that, clearly, $(\xi,\eta) \in \supp$ if
\[
\xi = (\xi_1,\xi_2^*,\xi_3^*,\ldots,\xi_{d+1}^*), \qquad \eta = (\xi_2^*,\xi_1,\xi_3^*,\ldots,\xi_{d+1}^*).
\]
The functional equation implies that $\widehat{f_1}(\xi_1)\widehat{f_2}(\xi_2^*) = \widehat{f_1}(\xi_2^*)\widehat{f_2}(\xi_1)$ for almost every $\xi_1 \in \mathbb{R}^d$, and hence $\widehat{f_1} = \lambda \widehat{f_2}$ for some $\lambda \neq 0$ (since each $f_j$ is assumed to be a nonzero function). In this way, we may argue that all $f_i$ are mutually parallel, and thus without loss of generality, we take $f_i = f$ for each $1 \leq i \leq d+1$.

Next, we argue that $\widehat{f}$ must satisfy the well-studied Maxwell--Boltzmann functional equation; that is,
\begin{equation} \label{e:MB}
\widehat{f}(\xi_1)\widehat{f}(\xi_2) = \widehat{f}(\eta_1)\widehat{f}(\eta_2)
\end{equation}
for almost every $\xi_1,\xi_2,\eta_1,\eta_2 \in \mathbb{R}^d$ such that $|\xi_1|^2 + |\xi_2|^2 = |\eta_1|^2 + |\eta_2|^2$ and $\xi_1 + \xi_2 = \eta_1 + \eta_2$. To see why this is the case, fix such $\xi_1,\xi_2,\eta_1,\eta_2$ and set
\[
\xi = (\xi_1,\xi_2,\xi_3^*,\ldots,\xi_{d+1}^*), \qquad \eta = (\eta_1,\eta_2,\xi_3^*,\ldots,\xi_{d+1}^*),
\]
where the $\xi_i^*$ are as above. Again it is easy to verify that $(\xi,\eta) \in \supp$, and therefore we immediately obtain \eqref{e:MB}.

At this point we may now appeal to the existing literature on the Maxwell--Boltzmann functional equation to deduce that $\widehat{f}(\xi) = e^{a|\xi|^2 + b \cdot \xi + c}$, where $\mathrm{Re}(a) < 0$, $b \in \mathbb{C}^d$ and $c \in \mathbb{C}$. For $\widehat{f} \in L^1(\mathbb{R}^d)$, this goes back at least to \cite{Lions} and \cite{Perthame} (see also \cite{Villani}); however the same conclusion holds for $\widehat{f}$ which are merely locally integrable (for details of this more general statement, see \cite{Foschi} for $d=2$, and \cite{BBJP} for arbitrary $d \geq 2$). Thus, it remains to justify that $\widehat{f}$ is locally integrable whenever the right-hand side of \eqref{fulltensor} is finite. However, this is clear from Proposition \ref{prop:carneirocomp}, since
\begin{align*}
\int_{(\mathbb{R}^d)^{d+1}} \prod_{i=1}^{d+1} |\widehat{f}(\xi_i)|^2I_{d,k}(\xi) \, d\xi & \geq c_{d,k} \sum_{m \neq n}  \int_{(\mathbb{R}^d)^{d+1}} \prod_{i=1}^{d+1} |\widehat{f}(\xi_i)|^2 |\xi_m - \xi_n|^{d^2 - kd - 2} \, d\xi \\
& = c_{d,k} \|\widehat{f}\|_2^{2(d-1)}  \int_{(\mathbb{R}^d)^{2}}  |\widehat{f}(\xi_1)|^2|\widehat{f}(\xi_2)|^2 |\xi_1 - \xi_2|^{d^2 - kd - 2} \, d\xi_1d\xi_2
\end{align*}
implies $\widehat{f} \in L^2(\mathbb{R}^d)$.
\end{proof}

\section{Appendix on covariance}  \label{sec:covar}
Here we establish the following fact which played a role in the proof of Proposition \ref{prop:weightchar}.
\begin{lemma}\label{lem:covar} Let $\var_\omega$ denote the covariance matrix of a vector-valued random variable taking the values $\{\omega_i\}_{i=1}^{d+1}$ with equal probability, and let
\[
M = I + (\sqrt{d+1}-1){\bf 11}^T.
\]
If $\Omega' \in \Mat(d,d)$ has $i$th column equal to $\omega_{i}-\omega_{d+1}$, then $\Omega_M\Omega_M^T=(d+1)\var_\omega$, where $\Omega_M = \Omega' M^{-1}$.
\end{lemma}

\begin{proof} Let $I\in\Mat(d,d)$ denote the identity matrix and define
 \[
 P_{d+1} = \left( \begin{array}{rcr}  & &  \\ & I &   \\
& & \\
-1 &\cdots & -1  \end{array} \right).
\]
It then follows that
\[
\Omega' = \Omega P_{d+1}
\]
where the $i$th column of $\Omega \in \Mat(d+1,d+1)$ is $\omega_i$, and therefore
\[
\Omega_M\Omega_M^T = \Omega P_{d+1} M^{-1} M^{-T} P_{d+1}^T \Omega^T.
\]
Since
\[
M^{-1} = I +\bigg(\frac{1}{\sqrt{d+1}}-1\bigg){\bf 11}^T
\]
it follows that
\[
P_{d+1}M^{-1}M^{-T} P_{d+1}^T =  I  - {\bf 11}^T
\]
and thus it remains to show $\frac{1}{d+1}\Omega( I - {\bf 11}^T)\Omega^T=\var_\omega$.

To see this, view $(\Omega, \mu)$ as a probability space, where  $\Omega$ is now the set of vectors that were the columns of the matrix $\Omega$, i.e. $\Omega=\{ \omega_i\}_{i=1}^{d+1}$, and let $\mu$ be the uniform probability measure. Let $X$ be a random variable taking values from $(\Omega, \mu)$ and let $\mathbf{E}(X)$ denote the expected value of $X$. As $X$ is a vector-valued random variable, we express its variance in a matrix. The $i$th diagonal element being the variance of the $i$th coordinate and the $(i,j)$th element being the covariance between the $i$th and $j$th coordinates. Specifically,
\[
\var_\omega = \mathbf{E}((X- \mathbf{E}(X))(X- \mathbf{E}(X))^T) = \mathbf{E}(XX^T) - \mathbf{E}(X)\mathbf{E}(X)^T .
\]
Now,
\[
\mathbf{E}(X) = \frac{1}{d+1}\sum_{i=1}^{d+1} \omega_i = \frac{1}{\sqrt{d+1}} \Omega {\bf 1}
\]
and consequently $\mathbf{E}(X)\mathbf{E}(X)^T = \frac{1}{d+1} \Omega{\bf 1 1}^T \Omega^T$. Direct computations for each component show that $\mathbf{E}(XX^T)=\frac{1}{d+1}\Omega\Omega^T$ and therefore
 \[
\var_\omega = \frac{1}{d+1}\Omega( I - {\bf 11}^T)\Omega^T,
 \]
as claimed.
\end{proof}


\begin{thebibliography}{MMM}
\bibitem{BBCRV} J. A. Barcel\'o, J. Bennett, A. Carbery, A. Ruiz, M. C. Vilela, \textit{Some special solutions of the Schr\"odinger equation}, Indiana Univ. Math. J. \textbf{56} (2007), 1581--1593.

\bibitem{BBJP} J. Bennett, N. Bez, C. Jeavons, N. Pattakos, \textit{On sharp bilinear
Strichartz estimates of Ozawa--Tsutsumi type}, J. Math. Soc. Japan \textbf{69} (2017), 459--476.

\bibitem{BCT} J. Bennett, A. Carbery, T. Tao, \textit{On the multilinear restriction and Kakeya conjectures}, Acta Math. \textbf{196} (2006), 261--302.

\bibitem{BS} J. Bennett, A. Seeger, \textit{The Fourier extension operator on large spheres and related oscillatory integrals}, Proc. Lond. Math. Soc. \textbf{98} (2009), 45--82.

\bibitem{BG} J. Bourgain, L. Guth, \textit{Bounds on oscillatory integral operators based on multilinear estimates}, Geom. Funct. Anal. \textbf{21} (2011), 1239--1295.

\bibitem{Cald} A. P. Calder\'on, \textit{On the Radon transform and some of its generalizations}, Conference on Harmonic Analysis in Honor of Antoni Zygmund, Vol. I, II, pp. 673--689 (1983).

\bibitem{Carneiro} E. Carneiro, \textit{A sharp inequality for the Strichartz norm}, Int. Math. Res. Not. \textbf{16} (2009), 3127--3145.

\bibitem{Christ} M. Christ, \textit{Extremizers of a Radon transform inequality}, Advances in Analysis: The Legacy of Elias M. Stein, Princeton University Press (2014), 84--107.

\bibitem{ChristQuil} M. Christ, R. Quilodr\'an, \textit{Gaussians rarely extremize adjoint Fourier restriction inequalities for paraboloids}, Proc. Amer. Math. Soc. \textbf{142} (2014), 887--896.

\bibitem{Drury84}  S. W. Drury, \textit{Generalizations of Riesz potentials and $L^{p}$ estimates for certain $k$-plane transforms}, Illinois J. Math. \textbf{28} (1984), 495--512.

\bibitem{Foschi} D. Foschi, \textit{Maximizers for the Strichartz inequality}, J. Eur. Math. Soc. \textbf{9} (2007), 739--774.

\bibitem{Gelbart} S. S. Gelbart, \textit{Fourier analysis on matrix space}, Memoirs of the American Mathematical Society, No. 108, American Mathematical Society, Providence, R.I., 1971.

\bibitem{Aimpaper} P. Gressman, D. He, V. Kova\v{c}, B. Street, C. Thiele, P. L. Yung, \textit{On a trilinear singular integral form with determinantal kernel}, Proc. Amer. Math. Soc. \textbf{144} (2016), 3465--3477.

\bibitem{GuthRest2} L. Guth, \textit{A restriction estimate using polynomial partitioning}, J. Amer. Math. Soc. \textbf{29} (2016), 371--413.

\bibitem{IR07} A. Iosevich,  M. Rudnev, \textit{Distance measures for well-distributed sets}, Discrete Comput. Geom. \textbf{38} (2007), 61--80.

\bibitem{Lions} P. L. Lions, \textit{Compactness in Boltzmann's equation via Fourier integral operators and applications, I}, J. Math. Kyoto Univ. \textbf{34} (1994), 391--427.

\bibitem{LR} R. Luc\`a, K. M. Rogers, \textit{Average decay of the Fourier transform of measures with applications}, arXiv:1503.00105, to appear in J. Eur. Math. Soc.

\bibitem{OS} D. M. Oberlin, E. M. Stein. \textit{Mapping properties of the Radon transform},  Indiana Univ. Math. J. \textbf{31} (1982), 641--650.

\bibitem{Perthame} B. Perthame, \textit{Introduction to the collision models in Boltzmann's theory}, Modeling of Collisions, 33 P.-A. Raviart, ed., Masson, Paris (1997).

\bibitem{PV} F. Planchon, L. Vega, \textit{Bilinear virial identities and applications}, Ann. Sci. \'{E}c. Norm. Sup\'{e}r. (4) \textbf{42}  (2009), 261--290.

\bibitem{Rubin06} B. Rubin, \textit{Riesz potentials and integral geometry in the space of rectangular matrices}, Adv. Math. \textbf{205} (2006), 549--598.

\bibitem{StrichartzKplane} R. S. Strichartz \textit{ $ L^ p $ estimates for {R}adon transforms in {E}uclidean and non-{E}uclidean spaces}, Duke Math. J. \textbf{48} (1981), 699--727.

\bibitem{VegaEsc} L. Vega, \textit{Bilinear virial identities and oscillatory integrals}, Harmonic analysis and partial differential equations, 219--232, Contemp. Math., 505, Amer. Math. Soc., Providence, RI, 2010.

\bibitem{VegaPAMS1988} L. Vega, \textit{Schr\"odinger equations: Pointwise convergence to the initial data}, Proc. Amer. Math. Soc. \textbf{102} (1988), 874--878.

\bibitem{Villani} C. Villani, \textit{Entropy methods for the Boltzmann equation}, Lecture Notes in Mathematics, Volume 1916, 2008, 1--70.
\end{thebibliography}
\end{document}